\documentclass[10pt,a4paper]{article}

\usepackage{stix}
\usepackage{mathrsfs}
\usepackage{amsmath}    
\usepackage{amsthm}
\usepackage[cp1250]{inputenc}                       
\usepackage{url}
\usepackage{mathtools}
\usepackage{amsfonts}
\usepackage{relsize}
\usepackage{amsxtra}
\usepackage{wrapfig}
\usepackage{tikz-cd}

\usepackage{tikz}
\usepackage{enumitem}

\newcommand{\N}{\mathbb{N}}

\newcommand{\dbarm}{\dbar_{\mathcal{M}}}
\newcommand{\fbarm}{\fbar_{\mathcal{M}}}

\newcommand{\h}{\bar{H}}
\newcommand{\sX}{\mathscr{X}}
\usepackage{thmtools}

\theoremstyle{plain}

\newtheorem{theorem}{Theorem}[section]
\newtheorem{lemma}[theorem]{Lemma}
\newtheorem{proposition}[theorem]{Proposition}

\newtheorem{rem}[theorem]{Remark}

\theoremstyle{definition}

\renewcommand{\subset}{\subseteq}

\newcommand{\blank}{\star}

\usepackage{array}       

\usepackage{fancyhdr}          
\usepackage{multicol}
\usepackage{lastpage}

\DeclareMathOperator{\Part}{Part}

\newcommand{\bbS}{\ensuremath{\mathbb{S}}}

\newcommand{\Z}{\ensuremath{\mathbb{Z}}}

\newcommand{\cP}{\mathcal{P}}
\newcommand{\cQ}{\mathcal{Q}}
\newcommand{\cE}{\mathcal{E}}
\newcommand{\cR}{\mathcal{R}}

\def \f {\bar f}

\def \und {\underline}
\def \N {\mathbb N}

\def \Z {\mathbb Z}

\def \eps {\varepsilon}

\newcommand{\dbar}{\bar d}
\newcommand{\fbar}{\bar f}
\newcommand{\X}{\mathbf{X}}

\DeclareMathOperator{\HSh}{\eta}

\renewcommand{\phi}{\varphi}

\title{Uniform continuity of entropy rate with respect to the $\fbar$-pseudometric}

\usepackage{blindtext,graphicx}
\usepackage[absolute]{textpos}
\begin{document}
 	\author{Tomasz Downarowicz
 		\and
 		Dominik Kwietniak
 		\and
 		Martha Łącka
 	}
 	
 	\newcommand{\Addresses}{{
 			\bigskip
 			\footnotesize
 			
 			T.~Downarowicz, \textsc{Faculty of Mathematics, Wroc{\l}aw University of Technology, Wybrze\.ze Wyspia\'nskiego 27,
 				50-370 Wroc{\l}aw, Poland}\par\nopagebreak
 			\textit{E-mail address}: T.~Downarowicz: \texttt{downar@pwr.wroc.pl}
 			\par\nopagebreak
 			\textit{Webpage}:  \url{http://prac.im.pwr.wroc.pl/~downar/}
 			
 			\medskip
 			
 			D.~Kwietniak, \textsc{Faculty of Mathematics and Computer Science, Jagiellonian University in Krakow, ul. \L o\-jasiewicza 6, 30-348 Krak\'ow, Poland}\par\nopagebreak
 			\textit{E-mail address}: D.~Kwietniak: \texttt{dominik.kwietniak@uj.edu.pl}
 				\par\nopagebreak
 				\textit{Webpage}:  \url{www.im.uj.edu.pl/DominikKwietniak/}
 				
 			\medskip
 			
 			M.~Łącka, \textsc{Faculty of Mathematics and Computer Science, Jagiellonian University in Krakow, ul. \L o\-jasiewicza 6, 30-348 Krak\'ow, Poland}\par\nopagebreak
 			\textit{E-mail address}: M.~Łącka: \texttt{martha.ubik@uj.edu.pl}	\par\nopagebreak
 			\textit{Webpage}:  \url{www2.im.uj.edu.pl/MarthaLacka/}
 			
 		}}
 		
 	\begin{textblock}{15}(1,1)
 		\noindent Accepted to \emph{IEEE Transactions on Information Theory, DOI: 10.1109/TIT.2021.3111831}
 	\end{textblock}
 	
\maketitle
\begin{abstract}
Assume that a sequence $x=x_0x_1\ldots$ is frequency-typical for a finite-valued stationary stochastic process $\X$.
We prove that the function associating to $x$ the entropy-rate $\h(\X)$ of $\X$ is uniformly continuous when one endows
the set of all frequency-typical sequences with the $\fbar$ pseudometric. As a consequence, we obtain the same result for the $\dbar$ pseudometric.
We also give an alternative proof
of the Abramov formula for the Kolmogorov-Sinai entropy of the induced measure-preserving transformation.
\end{abstract}
\section{Introduction}

Assume that  $\Lambda$ stands for a finite set and we are given two $\Lambda$-valued stationary stochastic processes, $\X=(X_i)_{i=0}^\infty$ and $\mathbf{Y}=(Y_i)_{i=0}^\infty$. Let $x=(x_i)_{i=0}^\infty$ and $y=(y_i)_{i=0}^\infty$ be frequency-typical realisations (samples) of, respectively $\X$ and $\mathbf Y$. Under what conditions on $x$ and $y$ can we conclude that the entropy rates $\h(\X)$ and $\h(\mathbf{Y})$
are close?

Note that the above question focuses on properties of individual frequency-typical trajectories to determine some global characteristics of stationary processes. This point of view was popularised by Shields \cite{Shields} and Weiss \cite{Weiss}, who presented central issues in ergodic theory and information theory in that ``sample path'' spirit. Here, we are interested in measurements of distortion determining a pseudometric such that endowing the space of all frequency-typical sequences with that pseudometric turns the entropy rate of the generated processes into a continuous function. Our problem is motivated by recent results in the dynamical systems theory, where the following construction obtains specific invariant measures: In the first step, one finds a  sequence of frequency-typical orbits converging in an appropriate sense. In the second step, one demonstrates that the limit of the approximating sequence is a frequency-typical orbit generating the sought measure. The question, whether the Kolmogorov-Sinai entropy of the limiting measure is the limit of entropies of measures generated by frequency-typical orbits in the approximating sequence reduces to the question stated in the first paragraph of our paper. By the nature of the construction, we work with concrete realisations (individual samples of the random processes) and pseudometrics $\dbar$ and $\fbar$ described below. We stress that we are looking for results valid for \emph{any} frequency-typical sample, while the so far existing results consider \emph{almost all} samples and often assume also ergodicity of the processes. 
We find the information-theoretic formulation a natural one for our problem.

By the \emph{entropy rate} of a stationary $\Lambda$-valued process $\mathbf{Z}=(Z_i)_{i=0}^\infty$  we mean
\begin{equation}\label{eq:process-entropy}
\h(\mathbf{Z})=\lim_{n\to\infty}\frac{\sum_{\lambda_1\ldots\lambda_{n}\in\Lambda^n}\HSh\left(\mu\left(Z_0=\lambda_1,\ldots,Z_{n-1}=\lambda_{n}\right)\right)}{n},
\end{equation}
where $\HSh$ stands for the entropy function given by $\HSh(0)=0$ and $\HSh(t)=-t\log t$ for $t>0$. For more details (in particular, for the justification that the limit exists) see \cite[Lemma 3.8]{Gray} or \cite[Sec. I.6.b]{Shields} or \cite{Downarowicz}.

To measure the distortion between sequences $x=(x_i)_{i=0}^\infty$ and $y=(y_i)_{i=0}^\infty$ over a common alphabet $\Lambda$ we first introduce a \emph{fidelity criterion}, that is a sequence $(\rho_n)_{n=1}^\infty$ of \emph{distortion measures} with $\rho_n$ defined on $\Lambda^n\times\Lambda^n$ for $n=1,2,\ldots$ and then take the limit superior as $n \to\infty$ of \emph{average (per-symbol) distortion} obtaining
\[
\bar{\rho}(x,y)=\limsup_{n\to\infty}\frac1n\rho_n(x_0x_1\ldots x_{n-1},y_0y_1\ldots y_{n-1}).
\]
This idea goes back to Shannon \cite{Shannon1, Shannon2}, see also \cite[Sec. 5.2, p. 120]{Gray}.

The simplest and most common example of the fidelity criterion  is based on a distortion measure known as the \emph{(additive) Hamming distance} given for $n\in\mathbb{N}$ and words $x_0x_1\ldots x_{n-1}$, $y_0y_1\ldots y_{n-1}$ in $\Lambda^n$ by
\[
d_n(x_0x_1\ldots x_{n-1},y_0y_1\ldots y_{n-1})=\left|\{0\le j< n: x_j\neq y_j\}\right|,
\]
which is the number of coordinates in which the sequences $x_0x_1\ldots x_{n-1}$ and $y_0y_1\ldots y_{n-1}$ differ. Note that the Hamming distance is  indeed  additive in the following sense: for every $n,m\in\mathbb{N}$ and $x_1x_2\ldots x_{n+m}$, $y_1y_2\ldots y_{n+m}$ in $\Lambda^{n+m}$ it holds
\begin{equation}\label{dbar-add}
  d_{n+m}((x_i)_{i=1}^{n+m},(y_i)_{i=1}^{n+m})= d_n((x_i)_{i=1}^{n},(y_i)_{i=1}^{n})+d_m((x_i)_{i=n+1}^{n+m},(y_i)_{i=n+1}^{n+m}).
\end{equation}

A less known fidelity criterion is based on a distortion measure provided by the sequence $(f_n)_{n=1}^\infty$ of $f$-distances (cf. \cite[p. 94]{Kalikow}). For $n\in\mathbb{N}$ the \emph{$f_n$-distance} between sequences $x_0x_1\ldots x_{n-1}$, $y_0y_1\ldots y_{n-1}$ in $\Lambda^n$  is simply the number of letters one must remove from each sequence so that the remaining words match, that is,
\[f_n(x_0x_1\ldots x_{n-1},y_0y_1\ldots y_{n-1})=n-k,\]
where $k$ is the largest integer such that for some $0\le i(1)<i(2)<\ldots<i(k)<n$ and $0\le j(1)<j(2)<\ldots<j(k)<n$ it holds $x_{i(s)}=y_{j(s)}$ for $s=1,\ldots,k$. The $f$-distance is sometimes called an \emph{edit distance}, since it depends on the number of edits (character deletions) that have to be performed in order to obtain matching sequences. Note that other functions are also known as \emph{edit distances} for example the \emph{Levenshtein distance}. The family of $f$-distances lacks the additivity property \eqref{dbar-add} which, vaguely speaking, is a source of difficulties when working with $\fbar$.

The $\dbar$ ($d$-bar) \emph{pseudometric} between sequences $x=(x_i)_{i=0}^\infty$ and $y=(y_i)_{i=0}^\infty$ in $\Lambda^{\mathbb{N}_0}$ (see also \cite{GNS, Ornstein}) is given by
\begin{equation}\label{def:dn}
         \dbar(x, y)=\limsup_{n\to\infty} \dbar_n(x_0x_1\ldots x_{n-1},y_0y_1\ldots y_{n-1}),
       \end{equation}
where $\dbar_n$ is  the \emph{average} (or \emph{per-letter}) \emph{Hamming distance} given by
\[
\dbar_n(x_0x_1\ldots x_{n-1},y_0y_1\ldots y_{n-1})=\frac1n d_n(x_0x_1\ldots x_{n-1},y_0y_1\ldots y_{n-1}).
\]
Note that $\dbar$ also appears in \cite[Section 5.3, p. 121]{Gray} under the name of \emph{sequence distortion}, where it is denoted by $\rho_\infty$.

Similarly, replacing $\dbar_n$ in \eqref{def:dn} by the \emph{average} (or \emph{per-letter}) \emph{$f$-distance}   $\fbar_n=\frac1n f_n$
we obtain the \emph{$\fbar$  ($f$-bar) pseudometric} on $\Lambda^{\mathbb{N}_0}$ defined (see also \cite[p. 92]{Kalikow}) for sequences $x=(x_i)_{i=0}^\infty$ and $y=(y_i)_{i=0}^\infty$ in $\Lambda^{\mathbb{N}_0}$ as
\begin{equation}\label{eq:def-fbar}
\fbar(x,y) =\limsup_{n\to\infty}\fbar_n(x_0x_1\ldots x_{n-1},y_0y_1\ldots y_{n-1}).
  \end{equation}
Clearly, for every pair of sequences $x=(x_i)_{i=0}^\infty$ and $y=(y_i)_{i=0}^\infty$ in $\Lambda^{\mathbb{N}_0}$ we have
\begin{equation}\label{ineq:fbar-le-dbar}
\fbar(x,y)\le \dbar(x,y).
\end{equation}
It is also easy to see that the pseudometrics given in \eqref{def:dn}  and \eqref{eq:def-fbar} are not equivalent, because taking $\Lambda=\{0,1\}$ and $x=(01)^\infty$, $y=(10)^\infty$ it holds that $\dbar(x,y)=1$, while $\fbar(x,y)=0$.

The pseudometrics $\dbar$ and $\fbar$ can be seen as sample sequence versions of the metrics between random processes, which unfortunately are also denoted by $\dbar$ and $\fbar$ (see \cite[Thm. 5.1]{Gray}, \cite[Def. 334, Def. 454]{Kalikow}, \cite[Def. 2.4]{ORW}, \cite[Def. 7.3]{Rudolph}, \cite[p. 92]{Shields}). In order to resolve this notational conflict, we will henceforth denote these distances between processes as $\dbarm$ and $\fbarm$.
The definition of $\dbarm$ is a variant of the construction of the Kantorovich (or the Kantorovich-Rubinstein, or the Wasserstein vel Vasershtein) optimal transport metric between two processes, where available transportation plans are shift-invariant (stationary) joinings of the processes (see \cite{Rudolph, Vershik, Villani}). Ornstein's $\dbarm$ metric plays a prominent role in the study of classification problem of Bernoulli processes \cite{Rudolph, Shields} and it is vital for information theory (Shannon coding theorems for stationary codes, universal coding, classifying noisy channels), see \cite{Gray}. The $\fbarm$ metric was introduced by Katok \cite{Katok} and Feldman \cite{Feldman}, and developed in a work of Ornstein, Rudolph and Weiss \cite{ORW}. It is crucial for the theory of Kakutani equivalence for Bernoulli and Kronecker systems \cite{ORW}. As mentioned above, our interest in the pseudometrics $\fbar$ and $\dbar$ on $\Lambda^{\mathbb{N}_0}$ comes from the fact that $\dbar$ and $\fbar$, as well as their topological counterparts (known as the Besicovitch and Feldman-Katok pseudometrics), proved to be very useful in constructions and exploration of stationary processes and invariant measures for continuous maps on compact metric spaces (see \cite{KL, KLO} and references therein).

Our main result states that the function which takes a frequency-typical sequence $z=(z_i)_{i=0}^\infty$  and associates to $z$ the entropy rate $\h(\mathbf Z)$ of the stationary process $\mathbf Z=(Z_i)_{i=0}^\infty$ generated by $z$ turns out to be uniformly continuous when we endow the set 
of all frequency-typical sequences in $\Lambda^{\mathbb{N}_0}$ with the pseudometric $\fbar$.

\begin{theorem}\label{thm:main-fbar}
For every finite alphabet $\Lambda$ and $\eps>0$ there is $\delta>0$ such that if $\X$ and $\X'$ are $\Lambda$-valued stationary processes and there exist frequency-typical sample sequences $x$ of $\X$ and $x'$ of $\X'$ satisfying $\fbar(x,x')<\delta$, then $|\h(\X)-\h(\X')|<\eps$.
\end{theorem}
As a corollary of the inequality \eqref{ineq:fbar-le-dbar} we immediately obtain an analogous result for $\dbar$.
\begin{theorem}\label{thm:main-dbar}
For every finite alphabet $\Lambda$ and $\eps>0$ there is $\delta>0$ such that if $\X$ and $\X'$ are $\Lambda$-valued stationary processes and there exist frequency-typical sample sequences $x$ of $\X$ and $x'$ of $\X'$ satisfying $\dbar(x,x')<\delta$, then $|\h(\X)-\h(\X')|<\eps$.
\end{theorem}

We were unable to find Theorem \ref{thm:main-fbar} in the present form in the literature, that is, without any assumptions on the ergodicity of the processes and assuming only the existence of a pair of $\fbar$-close frequency-typical sequences. Note that the continuity of the entropy-rate function as the function on the space of stationary processes endowed with $\dbarm$ is well-known, see \cite[Corollary 6.1]{Gray}, \cite[Theorem 385]{Kalikow}, or \cite[Thm. 7.9]{Rudolph} (although often stated only for ergodic processes, omitting uniform continuity as in \cite[Thm. I.9.16]{Shields}). The analogous statement for $\fbarm$ is known only for ergodic processes, see \cite[Theorem 455]{Kalikow} or \cite[Prop. 3.4]{ORW} and both sources use Abramov's formula for the Kolmogorov-Sinai entropy of induced transformation. To obtain Theorem \ref{thm:main-fbar} or Theorem \ref{thm:main-dbar} from the existing results about $\dbarm$ or $\fbarm$, one has to show that for every $\eps>0$ there is $\delta>0$ such that the existence of two frequency-typical sequences that are $\delta$ apart with respect to $\fbar$ (respectively, $\dbar$) on $\Lambda^{\mathbb{N}_0}$ implies that the corresponding processes are $\eps$ apart with respect to $\fbarm$ (respectively, $\dbarm$) metric on the space of processes. This is known for $\dbar$ and $\dbarm$, see  \cite[Theorem I.9.10]{Shields}, even without ergodicity \cite[Thm. 7.10]{Rudolph}, but known only for ergodic processes for $\fbar$ and $\fbarm$ (see \cite[Prop. 2.6 \& 2.7]{ORW}). Furthermore, the proof in \cite{ORW} uses the Shannon-McMillan-Breiman theorem for ergodic measures in a crucial way,  hence works only for ergodic processes.

In contrast, our demonstration of Theorem \ref{thm:main-fbar} works for not necessarily ergodic processes, contains the $\dbar$ case as a particular case,  and it is more direct even for $\dbar$  (it does not involve the use of auxiliary metrics $\dbarm$ or  $\fbarm$ on processes, does not require the Shannon-McMillian-Breiman theorem, nor the Abramov formula and conditional expectations). We use only the elementary properties of entropy. Our proof also immediately implies uniform continuity of the entropy rate function when the space of processes is endowed with the metric $\dbarm$. Again, we do not have to assume  ergodicity.

\begin{theorem}\label{thm:measure-dbar}
For every finite alphabet  $\Lambda$ and $\eps>0$ there is $\delta>0$ such that if $\X$ and $\X'$ are $\Lambda$-valued stationary processes and $\dbarm(\X,\X')<\delta$, then $|\h(\X)-\h(\X')|<\eps$.
\end{theorem}

This holds because for any two processes $X$ and $Y$ we can always find realisations $x$ and $x'$ such that $\dbarm(\X,\X')=\dbar(x,x')$ (this is easy for ergodic processes, see \cite{Shields}, for not necessarily ergodic processes it follows from \cite[Thm. 2.10]{CDS} and joining characterisation of $\dbarm$). Similarly, we obtain a new proof of the uniform continuity of the entropy rate function when the space of ergodic processes is endowed with the metric $\fbarm$. We have to restrict to ergodic processes, because the existence of sample sequences $x$ and $x'$ such that $\fbar(x,x')\le \fbarm(\X,\X')$ is known only for ergodic processes, see \cite[Prop. 2.5]{ORW}.

\begin{theorem}\label{thm:measure-fbar}
For every finite alphabet $\Lambda$ and $\eps>0$ there is $\delta>0$ such that if $\X$ and $\X'$ are ergodic $\Lambda$-valued stationary processes and $\fbarm(\X,\X')<\delta$, then $|\h(\X)-\h(\X')|<\eps$.
\end{theorem}

As a by-product of our approach, we obtain a new proof of the Abramov formula for the entropy of the induced transformation in general, not necessarily ergodic case (see Theorem \ref{thm:Abramov}). The result in such a generality (attributed to Scheller in \cite{Krengel}) is usually presented in the literature with an additional ergodicity assumption. Our demonstration requires only basic properties of the entropy conditioned on a countable partition. The usual proof uses conditional expectation and conditioning on $\sigma$-algebras. 

\section{Basic Facts and Notation}

\subsubsection*{Partitions and Names}
A \emph{measure preserving system} is a quadruple $(X,\mathscr{X},\mu,T)$, where $(X,\mathscr{X},\mu)$ is a standard probability space and $T\colon X\to X$ preserves $\mu$.
Let 
$\cP=\{P_\alpha:\alpha\in \Lambda\}$ be a measurable partition   of $X$
with $\Lambda\subset \mathbb{N}_0$. For $\alpha\in \Lambda$  we write $[\alpha]$  to denote $P_\alpha$ and refer to it as a \emph{cell} of the partition $\cP$.
The \emph{join} of two partitions  $\cP$, $\cQ$  of $X$ is the partition  $\cP\vee\cQ=\{[\alpha]\cap[\beta]:[\alpha]\in  \cP,\ [\beta]\in\cQ \}$. Since $\vee$ is associative, we can define the join of any finite collection of partitions (cf. \eqref{eq:PS}). We write $\cQ\succcurlyeq\cP$ if foe every $Q\in\cQ$ there is $P\in\cP$ such that $Q\subseteq P$. The \emph{partition distance} \cite{Gray} between $\cP$ and $\cQ$ is defined by
\[
|\cP-\cQ|=\sum_{\alpha\in\N_0} \mu(P_\alpha\textstyle{\mathsmaller{\triangle}} Q_\alpha)
\]
(we extend the alphabets if necessary by adding empty cells).
The \emph{(full) $\cP$-name} of $x\in X$ is a $\Lambda$-valued sequence $(x_n)_{n\in\bbS}$ such that for every $n\in\bbS$ we have that $x_n=\alpha$ if, and only if, $T^n(x)\in P_\alpha$. Given $S\subset\mathbb{R}$ with $S\cap\bbS$ finite, where  $\bbS=\Z$ if $T$ is invertible, and $\bbS=\N_0$ otherwise, we define
\begin{equation}\label{eq:PS}
\cP^S=\bigvee_{j\in S\cap\bbS}T^{-j}\cP=\left\{
\bigcap_{s\in S\cap\bbS} T^{-s}([\alpha_s]):(\alpha_s)_{s\in S\cap\bbS}\subseteq\cP
\right\}.\end{equation}
For $n\in\N$ we denote by $\cP^n$ the partition $\cP^{[0,\,n)}=\cP^{\{0,1,\ldots,n-1\}}$.
Note that $\cP^1=\cP$. Cells of $\cP^n$ correspond to finite $\Lambda$-valued strings of length $n$, hence for $n\ge 1$ and $\alpha_0,\alpha_1,\ldots,\alpha_{n-1}\in \Lambda$ we write
\[
[\alpha_0\alpha_1\ldots\alpha_{n-1}]=[\alpha_0]\cap T^{-1}([\alpha_1])\ldots\cap T^{-n+1}([\alpha_{n-1}])\in\cP^n.
\]
Similarly, the cells of $\cP^{[1,n]}$ (for $n\ge 1$) consist of points sharing $\cP$-name for entries from $1$ to $n$, hence we denote them as
\[
[\blank\alpha_1\ldots\alpha_{n}]=\{x\in X:T^j(x)\in[\alpha_j]\text{ for $j=1,\ldots,n$}\},\text{ where $\alpha_1,\ldots,\alpha_{n}\in \Lambda$}.
\]
We have used ``$\blank$'' to stress that we do not know which symbol appears at the $0$ coordinate in the $\cP$-name of a point from a cell of $\cP^{[1,n]}$. 
We clearly have
\[
[\alpha_0]\cap [\blank\alpha_1\ldots\alpha_{n}]=[\alpha_0\alpha_1\ldots\alpha_{n}]\text{ for all $\alpha_0,\alpha_1,\ldots,\alpha_{n}\in \Lambda$}.
\]
We will also consider partitions of $X$ according to the entries in the $\cP$-names of points over blocks of varying length.
Assume that $\xi\colon X\to\N$ is a measurable function with $\int_X\xi\,\text{d}\mu<\infty$, and $\cP$ is a finite partition of $X$. We define $\cP^{[1,\xi]}$ to be the partition obtained as follows. First, we partition $X$ into level sets of $\xi$, that is we take $\Xi=\{\xi^{-1}(n): n\in\N\}$. Second, for every $n\ge 1$ we further partition the set $\xi^{-1}(n)$ of $\Xi$ according to $\cP^{[1,n]}$. Each cell of $\cP^{[1,\xi]}$ gathers points sharing the $\cP$-name from time $1$ to $n$ where $n$ is the common value of $\xi$ for all these points. 
That is,
\[
\cP^{[1,\xi]}=
\bigcup_{n=1}^\infty \{P\cap \xi^{-1}(n):P\in \cP^{[1,n]}\}.
\]
Equivalently, for $x\in X$ the cell of $\cP^{[1,\xi]}$ containing $x$ coincides with the cell of $\cP^{[1,\xi(x)]}$ containing $x$. We can extend this notation in an obvious way, and define 
$\cP^{(-\xi,0]}$. 

\subsubsection*{Entropy and Conditional Entropy of a Partition}
Let $(X,\sX,\mu)$ be a probability space and $\cP$, $\cQ$, and $\cR$ be countable measurable partitions of $X$.
The \emph{entropy} of $\cP$ is
\[H_{\mu}(\cP)=-\sum_{P\in\cP}\mu(P)\log\mu(P).\]
The \emph{conditional entropy} of $\cP$ given $\cQ$ is defined by
\[
H_\mu(\cP|\cQ) = \sum_{Q\in\cQ}\mu(Q) H_{\mu_Q}(\cP),
\]
where $\mu_Q$ is the conditional probability measure on $Q$ (that is the measure obtained by restricting $\mu$ to $Q$ and normalizing it). Clearly, $H_\mu(\cP)=H_\mu(\cP|\{X\})$. We note the following monotonicity properties of the entropy (see \cite[Lemma I.6.6]{Shields}):
\begin{gather}\label{ent-monotone-1}
  \cP\succcurlyeq\cQ\implies H_\mu(\cQ|\cR)\le H_\mu(\cP|\cR), \\
  \label{ent-monotone-2}
  \cP\succcurlyeq\cQ\implies H_\mu(\cR|\cP)\le H_\mu(\cR|\cQ).
\end{gather}

\subsubsection*{Stationary Processes and Measure Preserving Systems}
By a ($\Lambda$-valued) \emph{random variable} we mean a measurable function from a standard probability space $(\Omega,\mathscr{B},\nu)$ to a set 
$\Lambda\subset \N_0=\{0,1,2,\ldots\}$ endowed with the power set $\sigma$-algebra $\mathscr{P}(\Lambda)$. We also refer to $\Lambda$ as to an \emph{alphabet}. A  $\Lambda$-valued \emph{process} is a sequence of $\Lambda$-valued random variables $\X=(X_i)_{i\in\bbS}$, where $\bbS=\Z$ or $\bbS=\N_0$ is an index
set, such that the domain of each $X_i$ is a common standard probability space $(\Omega,\mathscr{B},\nu)$.  The process $\X$ is \emph{stationary} if for every
$n\ge 1$, every $\lambda_1,\ldots,\lambda_n\in\Lambda$, and any $s\in\bbS$ we have \begin{multline*}
  \nu\left(\left\{\omega\in\Omega:X_{j-1}(\omega)=\lambda_j\text{ for }j=1,\ldots,n\right\}\right)= \\
  = \nu\left(\left\{\omega\in\Omega:X_{s+j-1}(\omega)=\lambda_j\text{ for }j=1,\ldots,n\right\}\right).
\end{multline*}
We call $\nu$ the \emph{law} of $X$.

Processes and probability preserving systems are closely connected, see \cite[Sec. I.2]{Shields} or \cite[Sec. 1.3]{Gray}. We briefly recall that connection.

Given a $\Lambda$-valued stationary process $\X=(X_i)_{i\in\bbS}$, 
$n\ge 1$, $\lambda_1,\ldots,\lambda_n\in\Lambda$, and $t(1),\ldots,t(n)\in\bbS$ we set 
\begin{multline}\label{eq:mu}
\mu\left(\left\{x\in\Lambda^\bbS:x_{t(j)}=\lambda_j\text{ for }j=1,\ldots,n\right\}\right) =   \\
     = \nu\left(\left\{\omega\in\Omega:X_{t(j)}(\omega)=\lambda_j\text{ for }j=1,\ldots,n\right\}\right),
\end{multline}
where $(\Omega,\mathscr{B},\nu)$ is the standard probability space on which all $X_i$'s are defined. We easily see that $\mu$ extends to a probability measure on a $\mu$-completion $\mathscr{C}$ of the product $\sigma$-algebra on $\Lambda^\bbS$. Furthermore, since $\X$ is stationary, we conclude that $\mu$ is invariant for the shift transformation $\sigma\colon\Lambda^\bbS\to\Lambda^\bbS$ given for $x=(x_i)_{i\in\bbS}$ by $\sigma(x)_i=x_{i+1}$ for every $i\in\bbS$. It follows that $(\Lambda^\bbS,\mathscr{C},\mu,\sigma)$ is a probability preserving system, called a \emph{shift system}. Furthermore, $\sigma$ is invertible provided $\bbS=\mathbb{Z}$.

On the other hand, assume we have a measure preserving system $(X,\mathscr{X},\mu,T)$ and a measurable partition  $\cP=\{P_\alpha:\alpha\in \Lambda\}$ of $X$
where $\Lambda\subset \mathbb{N}_0$. 
We tacitly ignore $\mu$-null cells.  
Given $x\in X$ we write $\cP(x)=\alpha$ if, and only if, $x\in [\alpha]$. We construct a process $(T,\cP)$ by defining random variables $X_i=\cP\circ T^i$ for each $i\in \bbS$, where  $\bbS=\Z$ if $T$ is invertible, and $\bbS=\N_0$ otherwise. Then all $X_i$'s are $\Lambda$-valued random variables defined over the standard probability space $(X,\mathscr{X},\mu)$, and the sequence $\X=(X_i)_{i\in\bbS}$ is a stationary process. The entropy rate of the process $(T,\cP)$ is denoted  by $\h_\mu(T,\cP)$, that is $\h_\mu(T,\cP)=\h(\X)$, where 
$\h(\X)$ is given by \eqref{eq:process-entropy}. The quantity $\h_\mu(T,\cP)$ appears as \emph{(dynamical) entropy of $\cP$} in ergodic theory literature. Note that
\begin{equation}\label{eq:future}
\h_\mu(T, \cP)=\lim_{n\to\infty}\frac1n{H_\mu(\cP^n)}=\lim_{n\to\infty}H_\mu(\cP|\cP^{[1,n]}),
\end{equation}
where the first equality follows immediately from definitions and the second follows from \cite[Lemma 3.17]{Gray}.

In particular, if a measure preserving system $(X,\mathscr{X},\mu,T)$ is a shift system over the alphabet $\Lambda$, that is, $X=\Lambda^\bbS$ for some finite set $\Lambda$, $\mathscr{X}$ is the $\mu$-completion of   the product $\sigma$-algebra on $\Lambda^\bbS$, and $\mu$ is a $\sigma$-invariant probability measure on $\Lambda^\bbS$, then the associated process is obtained by taking  the partition $\mathcal{L}=\{[\lambda]:\lambda\in \Lambda\}$, where $[\lambda]=\{x\in\Lambda^\bbS: x_0=\lambda\}$ for $\lambda\in\Lambda$. The entropy rate of the process $(\sigma,\mathcal{L})$ is simply denoted as $h_\mu(\sigma)$. Note that a Cartesian product of a pair of shift systems is a shift system over an Cartesian product of their alphabets. Nevertheless, if $\zeta$ is an invariant measure of a shift system whose alphabet is a Cartesian product, we will write $h_\zeta(\sigma\times\sigma)$ for its entropy rate.
\subsubsection*{Frequency-typical sequences}
Let $(k_n)_{n=1}^\infty$ be an increasing sequence in $\mathbb{N}_0$. A $\Lambda$-valued sequence $z=(z_i)_{i\in\bbS}$ is \emph{frequency-typical along $(k_n)_{n=1}^\infty$} if   for every $m\in\mathbb{N}$ and $\lambda_1\ldots\lambda_{m}\in\Lambda^m$
the sequence
\begin{equation}\label{eq:freq-limit}
  \frac{1}{k_n}\left|\{0\le j \le k_n-m: z_j=\lambda_1,\ldots,z_{j+m-1}=\lambda_{m}\}\right| 
\end{equation}
converges as $n\to\infty$. 
We say that $z$ is \emph{frequency-typical (along $(k_n)_{n=1}^\infty$) for} or \emph{generates (along $(k_n)_{n=1}^\infty$)}
a stationary $\Lambda$-valued stochastic process $\mathbf Z=(Z_i)_{i\in\bbS}$ with the law $\nu$ if for every $m\ge 1$ and for every $\lambda_1,\ldots,\lambda_{m}\in\Lambda$ the sequence in \eqref{eq:freq-limit} converges to $\nu\left(Z_0=\lambda_1,\ldots,Z_{m-1}=\lambda_{m}\right)$ as $n\to\infty$. In that case, we also say that $z$ is \emph{frequency-typical (along $(k_n)_{n=1}^\infty$) for} or \emph{generates (along $(k_n)_{n=1}^\infty$)} the $\sigma$-invariant measure $\mu$ on $\Lambda^\bbS$ obtained from $\nu$ by \eqref{eq:mu}. Every frequency-typical sequence generates a unique stationary process and every stationary process has a frequency-typical sequence generating it. Furthermore, frequency-typical sequences have full measures for ergodic processes (for definition, see \cite[Sec. I.2.a]{Shields}) and form a null set for non-ergodic processes. We skip  ``(along $(k_n)_{n=1}^\infty$)'' whenever one can take $k_n=n$ for every $n\in\mathbb{N}$ in the above definitions. When we want to say that ``$z$ generates (along $(k_n)_{n=1}^\infty$) a $\sigma$-invariant measure $\mu$ on $\Lambda^\bbS$'' without mentioning $(k_n)_{n=1}^\infty$ explicitly, we simply say that \emph{$z$ quasi-generates $\mu$}. We will use the following observation without further reference: Given an increasing sequence $(k_n)_{n=1}^\infty$ in $\mathbb{N}_0$ and $z\in \Lambda^\bbS$ we can always find a subsequence of $(k_n)_{n=1}^\infty$ such that $z$ is frequency-typical along that subsequence for a $\sigma$-invariant measure on $\Lambda^\bbS$.

\subsubsection*{Induced measure preserving systems}
Given a measure preserving system $(X,\sX,\mu,T)$ and $E\subset X$ with $\mu(E)>0$ we write $\sX_E$ for the trace $\sigma$-algebra on $E$, that is $\sX_E=\{A\cap E: A\in\sX\}$. The induced measure on $E$ is $\mu_E$, where $\mu_E(A)=\mu(A)/\mu(E)$ for $A\in\sX_E$. The \emph{first return time} to $E$ is the measurable function $r_E\colon E\to \bar\N=\N\cup\{\infty\}$ defined by $r_E(x)=\min\{n\ge 1: T^n(x)\in E\}$. By the Poincar\'e recurrence theorem $r_E(x)<\infty$ for $\mu$-a.e. $x\in E$. Setting $R_n=\{x\in E: r_E(x)=n\}$ for $n\in\bar{\mathbb{N}}$, we obtain the \emph{return time partition}  $\cR$ of $E$. The Kac lemma (see \cite{Saussol} for a proof in this generality) says that
\begin{equation}\label{eq:Kac}
  \int_E r_E\, \text{d}\mu=\sum_{n\in\N}n\mu_E(R_n)=\frac{\mu\left(\bigcup_{k}T^{-k}(E)\right)}{\mu(E)}.
\end{equation}
The \emph{induced map} is the map $T_E\colon E\to E$ such that $T_E(x)=T^{r_E(x)}(x)$ for $x\in E$  (we ignore $\mu$-null set of points where $T_E$ is not well-defined)).  The induced system is the measure preserving system $(E,\sX_E,\mu_E,T_E)$. We say that $E\in\sX$ \emph{sweeps out $X$} if
\begin{equation}\label{def:sweep-out}
  \mu\left(\bigcup_{n=0}^\infty T^{-n}(E)\right)=1.
\end{equation}
		Note that if $E$ is sweeps out $X$, then the trace $\sigma$-algebra $\sX_{E}$ together with $T$ contain complete information about $\sX$. In other words, for every $A\in\sX$ one can find sets $A_n\in\sX_{E}$ such that $A$ is a union (up to measure zero) of the sets $T^{-n}(A_n)$, $n\in\N_0$ (see \cite[p. 263]{DGS}).
If $\mu$ is ergodic, then every $A\in\sX$ with $\mu(A)>0$ sweeps out $X$.


\section{Main Lemmas}

In this section we state two lemmas (Propositions \ref{prop} and \ref{prop:main}) providing a backbone for the proof of our main result (Theorem \ref{thm:main-fbar}).
We postpone the proofs to the next section and the appendix.

Roughly speaking, the lemmas compare the entropy rates of processes obtained by encoding a measure preserving system $(X,\sX,\mu,T)$ and its induced first-return measure preserving system $(E,\sX_E,\mu_E,T_E)$. We choose such partitions $\cP$ of $X$ and $\cP_E$ of $E$ that knowing either the $\cP$-name of a $T$ orbit or the $\cP_E$-name of a corresponding $T_E$ orbit allows us to recover the other name. Note that the entropy rate is the expected number of symbols per unit time needed to encode an orbit in a measure preserving system by the coding algorithm induced by the transformation and the partition. More precisely, the entropy rate is the infimum of the mean number of symbols per unit time per point decoding the long-time behavior of a large batch of points chosen randomly according to the underlying measure (here, either $\mu$ or $\mu_E$). For the induced transformation, the ``unit time'' is the expected time of the next visit to $E$, that is, equals  $1/\mu(E)$. Now, two different encodings of the same orbit contain the same information so the corresponding entropy rates change inverse proportionally to the scaling of time unit (e.g. the number of megabits per minute is 60 times the number of megabits-per-second). 

From now on we assume that $(X,\mathscr{X},\mu,T)$ is an invertible probability measure preserving system. This is not a restrictive assumption, because in the measure preserving (stationary) setting the notion of natural extension allows us to transfer the results presented here to the noninvertible case.
Then $\mu$-almost every point visits $E$ for infinitely many positive and for infinitely many negative times.  We set $\cE$ to be the \emph{entry time partition of $E$}, that is, the cells of $\cE$ are $E_n=\{y\in E:r_E(T_E^{-1}(y))=n\}$ for $n\in\bar\N$. One can easily see that $\cE=T_E(\cR)$.
Furthermore,
$\mu_E(E_n)=\mu_E(R_n)$ for every $n\in\N$, so 
applying \eqref{eq:Kac} we get $H_{\mu_E}(\cE)=H_{\mu_E}(\cR)<\infty$. If $\cQ=\{Q_\alpha:\alpha\in A'\}$ is any partition of $E$ such that $\cQ\succcurlyeq \cE$, then the \emph{height} $|\alpha|$ of  $[\alpha]\in\cQ$ is $n\in \N$ such that $[\alpha]\subset E_n$.

Given a partition $\cP$ of $X$ we let $\cP_E$ stand for the partition of $E$ given by $\cP_E=\{P_\alpha\cap E:P_\alpha\in\cP\}$.
Our first lemma, Proposition \ref{prop:main} allows us to compare the entropy rate of a process determined by the induced transformation and a partition $\cP_E$ of $E$ such that $\cP_E\succcurlyeq \cE$ with a process determined by $T$ and a partition $\cP$ extending $\cP_E$ to the whole $X$ by adjoining the cell $P_0=X\setminus E$. This result was first stated in Scheller's thesis (see \cite{Krengel}, \cite[p. 257--259]{Petersen}), but our proof is new. 
Note that if  $P_0=X\setminus E\in\cP$ and $\cP_E\succcurlyeq \cE$, then knowing the $(T_E,\cP_E)$-name $(a_n)$ of $x\in E$ we determine the $(T,\cP)$-name of $x$  by inserting $|a_n|-1$ $0$'s between $a_{n-1}$ and $a_n$ for each $n\in\Z$.

\begin{proposition}\label{prop:main}
Assume that $(X,\mathscr{X},\mu,T)$ is an invertible probability measure preserving system and $E\in\sX$ sweeps out $X$. 
If $\cP$ is a  countable measurable partition of $X$ such that $H_\mu(\cP)<\infty$, $X\setminus E\in \cP$, and
$\cP_E \succcurlyeq\cE$, then $\mu(E)h_{\mu_E}(T_E,\cP_E)= h_{\mu}(T,\cP)$.
\end{proposition}

Consider a finite partition $\cQ=\{Q_s:s\in\Lambda\}$ of $X$, where $\Lambda\subseteq\N$. We assume that $\cQ$ contains the information whether the orbit is in or outside $E$, that is, $\cQ\succcurlyeq\{E,X\setminus E\}$. We want to find a partition of $E$ such that the process generated by the partition and $T_E$ encodes the full information about the process $(T,\cQ)$. We achieve this by adding to  the cells of $\cQ$  contained in $E$ the information about entry times.

It is now customary to think of a cell $[\alpha]=P_\alpha\in\cQ$ contained in $E$ as represented or indexed by a \emph{starred symbols} $\alpha^*$, while the cells of $\cQ$ outside $E$ are indexed as before by $\alpha\in\Lambda$. At least one starred and at least one non-starred symbol should index a nonempty set.
Since $\mu$-almost every $x\in X$ visits $E$ infinitely many times, the $\cQ$-name $(x_n)_{n\in\mathbb{Z}}$ of $\mu$ almost every $x\in X$ can be divided into blocks $x_{(n_{k-1},n_k]}$ where $n_0$ is chosen so that $n_0$ is the time of the first visit of $x$ to $E$, that is, $n_0=\min\{j\ge 0:T^j(x)\in E\}$. Each word $x_{(n_{k-1},n_k]}$ consists of some number  (possibly zero when $E\cap T(E)\neq\emptyset$) of non-starred symbols followed by a single starred one. Now, if we consider the words $x_{(n_{k-1},n_k]}$ as symbols of a new alphabet, we obtain a countable partition $\cQ_E^{(-\xi, 0]}$ of $E$, where $\xi=r_E\circ T_E^{-1}$. Recall that cells of $\cQ_E^{(-\xi, 0]}$ are defined by	taking the \emph{entry-time} partition $\cE$ of $E$ and refining each $E_n\in\cE$ according to $\cQ^{(-n,0]}$.
We will show that $\cQ^{(-\xi,0]}_0=\cQ_E^{(-\xi,0]}\cup\{X\setminus E\}$ is a partition of $X$ satisfying the assumptions of Proposition~\ref{prop:main}, which yields $\h_\mu(T,\cQ^{(-\xi,0]}_0)=\mu(E) \h_{\mu_E}(T_E,\cQ_E^{(-\xi,0]})$. On the other hand, the processes $(T,\cQ)$ and $(T,\cQ^{(-\xi,0]}_0)$ are isomorphic by a code which turns a $\cQ^{n_k-n_{k-1}}$-name $x_{(n_{k-1},n_k]}$ into $\cQ_0^{(-\xi, 0]}$-name $0^{n_k-n_{k-1}-1}w$ with $w=x_{(n_{k-1},n_k]}$ treated as a symbol of $\cQ_E^{(-\xi, 0]}$.


	\begin{proposition}\label{prop} Assume that $(X,\mathscr{X},\mu,T)$ is an invertible probability measure preserving system and $E\in\sX$ sweeps out $X$.
		If $\cQ$ is a finite partition of $X$ such that $\cQ\succcurlyeq\{E,X\setminus E\}$, then the partition $\cQ_E^{(-\xi, 0]}$ of $E$, where $\xi=r_E\circ T_E^{-1}$
satisfies $H_{\mu_E}(\cQ_E^{(-\xi, 0]})<\infty$ and $\mu(E) \h_{\mu_E}(T_E,\cQ_E^{(-\xi, 0]})=\h_\mu(T,\cQ)$.
	\end{proposition}

	\section{Proofs}
	
	\begin{proof}[Proof of Proposition~\ref{prop:main}]
For $x\in X$ and $N\in\mathbb{N}$ we inductively define an auxiliary function we call the \emph{time of the $N$-th return to $E$} and denote $v^N(x)$.
Let $v^1(x)=\min\{n\ge 1: T^n(x)\in E\}$ (we agree that $\min\emptyset=\infty$). Given $N>1$ and $v^{N-1}(x)\neq\infty$, we set
$v^N(x)=\min\{n> v^{N-1}(x): T^n(x)\in E\}$. By convention, we do not define $v^N(x)$ if $v^M(x)=\infty$ for some $M<N$.
Note that by \eqref{def:sweep-out} for $\mu$-almost every $x\in X$ we  have that for each $N\in\mathbb{N}$ it holds $N\le v^N(x)<\infty$. Furthermore, $v^1$ coincides with $r_E$ on $E$.

Fix $N\in\N$. For $n\ge N$ we have $\cP^{[1,v^n]}\succcurlyeq\cP^{[1,n]}\succcurlyeq\cP^{[1,(v^N\wedge n)]}$, where 
$(v^N\wedge n)(x)=\min\{v^N(x),n\}$ for $x\in X$.
		Using \eqref{ent-monotone-2} we obtain
		\begin{equation}\label{ineq:cond}
		H_{\mu}(\cP|\cP^{[1,v^n]})\le H_{\mu}(\cP|\cP^{[1,n]})\le H_{\mu}(\cP|\cP^{[1,(v^N\wedge n)]}).
		\end{equation}
		
		Note that $\cP^{[1,v^N]}$ and $\cP^{[1,(v^N\wedge n)]}$ coincide outside the set $\{x\in X: v^N(x)>n\}$, whose $\mu$-measure goes to $0$ as $n\to \infty$. It follows that the partition distance  $|\cP^{[1,v^N]}-\cP^{[1,(v^N\wedge n)]}|$ approaches $0$ as $n\to\infty$. 
		By \cite[Fact 1.7.10]{Downarowicz}, $H_\mu(\cP|\cdot)$ is continuous on the space of countable partitions endowed with the partition distance, so
		\begin{equation}\label{eq:Dow1710}
		\lim_{n\to\infty}H_{\mu}(\cP|\cP^{[1,(v^N\wedge n)]})=H_{\mu}(\cP|\cP^{[1,v^N]}).
		\end{equation}
		Let $n\to\infty$ in \eqref{ineq:cond}, then we use \eqref{eq:Dow1710}, and finally we let $N\to \infty$ to get
		\begin{equation}
		\label{eq:first-half-of-the-proof}
		\lim_{n\to\infty} H_{\mu}(\cP|\cP^{[1,v^n]})= \lim_{n\to\infty} H_{\mu}(\cP|\cP^{[1,n]}).
		\end{equation}
		By the definition of the conditional entropy we have
		\begin{equation}\label{eq:cond-ent}
		H_{\mu}(\cP|\cP^{[1,v^n]})= \sum_{W\in \cP^{[1,v^n]}}\mu(W)\sum_{A\in\cP}\HSh(\mu(A\cap W)/ \mu(W)).
		\end{equation}
		Below, we label the cells of $\cP_E\subset\cP$ with boldface letters  (recall that $[0]=X\setminus E\in \cP$). Note that $\cP_E\succcurlyeq \cE$, so if $[\mathbf{a}]\in\cP_E$, then every occurrence of $\mathbf{a}$ in the $\cP$-name 
		is preceded by $0^{|\mathbf{a}|-1}$.   
		We write $\mathbf{\bar{0}a}$ for blocks of ``correct'' number of $0$'s followed by $\mathbf{a}$, that is, $\mathbf{\bar{0}a}$ stands for $0^{|\mathbf{a}|-1}\mathbf{a}$. Cells of $\cP^{[1,v^n]}$ consist of points sharing $\cP$-names from position $1$ to the place where the $n$-th bold symbol occurs. In particular, if $W\in\cP^{[1,v^n]}$ then 
		\begin{equation}\label{eq:W}
		W=[\blank 0^{\ell}\mathbf{a}_1 {\mathbf{\bar{0}}}\mathbf{a}_2\ldots{\mathbf{\bar{0}}}\mathbf{a}_n]\in\cP^{[1,\ell+|\mathbf{a}_2|+\ldots+|\mathbf{a}_n|]}, \text{ where $\ell<|\mathbf{a}_1|$}.
		\end{equation}
		We have two cases: either $\ell=|\mathbf{a}_1|-1$ or $\ell<|\mathbf{a}_1|-1$.
		In the former case, we call $W$ a \emph{full} cell and note that $W\subset E$, so $\mu([0]\cap W)=0$. In the latter case, $W\subset X\setminus E$ and $\mu([0]\cap W)=1$. 
		Let $\mathcal{F}_n$ be the set of full cells in $\cP^{[1,v^n]}$.   
		We rewrite \eqref{eq:cond-ent} as
		\begin{equation}
		\label{eq:cond-ent2}
		H_{\mu}(\cP|\cP^{[1,v^n]})=\\
		\sum_{[\blank{\mathbf{\bar{0}}}\mathbf{a}_1\!\ldots\!{\mathbf{\bar{0}}}\mathbf{a}_n]\in\mathcal{F}_n}\!
		\mu([\blank{\mathbf{\bar{0}}}\mathbf{a}_1\!\ldots\! {\mathbf{\bar{0}}}\mathbf{a}_n])\sum_{[\mathbf{a}_0]\in\cP}\!\HSh\left(\frac{\mu([\mathbf{a}_0 {\mathbf{\bar{0}}}\mathbf{a}_1\!\ldots\!{\mathbf{\bar{0}}}\mathbf{a}_n])}{ \mu([\blank{\mathbf{\bar{0}}}\mathbf{a}_1\!\ldots\!{\mathbf{\bar{0}}}\mathbf{a}_n])}\right).
		\end{equation}
		Note that a full cell $[\blank \mathbf{\bar{0}}\mathbf{a}_1\!\ldots\!\mathbf{\bar{0}}\mathbf{a}_n]\subset\cP^{[1,v^n]}$ equals $[\blank \mathbf{a}_1\! \ldots\! \mathbf{a}_N]_E\in\cP_E^{[1,n]}$, where
		$\cP_E^{[1,n]}=T_E^{-1}(\cP_E)\vee \ldots \vee T_E^{-n}(\cP_E)$ and $[\mathbf{a}_0]\cap [\blank \mathbf{\bar{0}}\mathbf{a}_1\!\ldots\!\mathbf{\bar{0}}\mathbf{a}_n] =[\mathbf{a}_0 \mathbf{a}_1\!\ldots\! \mathbf{a}_n]_E\in\cP_E^{n+1}$ for every $[\mathbf{a}_0]\in\cP_E$.
		Since $\mu(B)=\mu(E)\mu_E(B)$ for $B\subset E$ , the right hand side of \eqref{eq:cond-ent2} equals
		\[\sum_{[\!\blank\mathbf{a}_1\!\ldots\!\mathbf{a}_n]\in \cP_E^{[1,n]}}\mu([\!\blank\mathbf{a}_1\!\ldots\!\mathbf{a}_n]_E)\sum_{[\mathbf{a}_0]\in\cP_E}\!\HSh\left(\frac{\mu_E([\mathbf{a}_0 \mathbf{a}_1\!\ldots\!\mathbf{a}_n]_E)}{ \mu_E([\!\blank\mathbf{a}_1\!\ldots\!\mathbf{a}_n]_E)}\right)=\mu(E)H_{\mu_E}\!(\cP_E|\cP_E^{[1,n]}).
		\]
		Letting $n\to \infty$, invoking \eqref{eq:future} 
and \eqref{eq:first-half-of-the-proof} we obtain
		\[
		\mu(E)\h_{\mu_E}\!(T_E,\cP_E)=\mu(E)\lim_{n\to\infty}H_{\mu_E}\!(\cP_E|\cP_E^{[1,n]}) =\lim_{n\to\infty}H_{\mu}(\cP|\cP^{[1,v^n]})=\h_\mu(T,\cP),
		\]
		which completes the proof.
	\end{proof}
		\begin{proof}[Proof of Proposition~\ref{prop}]
Since every partition $\cP=\{P_1,P_2,\ldots\}$ with $\sum_n n\mu(P_n)<\infty$ has finite entropy, so $H_{\mu_E}(\cE)=H_{\mu_E}(\cR)<\infty$.
		Let $\xi=r_E\circ T_E^{-1}$.	 
Observe that for every $n$ there are at most $|\cQ|^n$ atoms of $\cQ_E^{(-\xi, 0]}$ with nonempty intersection with $E_n$, that is, atoms corresponding to $\cQ^n$-names with the last symbol starred. It follows that for every $n\ge 1$ the entropy of the partition $\cQ_E^{(-\xi, 0]}$ with respect to the measure
			$\mu_E(\cdot\cap E_n)/\mu_E(E_n)$ is at most $n\log|\cQ|$. It follows directly from the way we defined $\cQ_E^{(-\xi, 0]}$ that $\cQ_E^{(-\xi, 0]}\succcurlyeq \cE$, hence $\cQ_E^{(-\xi, 0]}\vee \cE=\cQ_E^{(-\xi, 0]}$ and
			\[
			H_{\mu_E}(\cQ_E^{(-\xi, 0]})=H_{\mu_E}(\cQ_E^{(-\xi, 0]}|\cE)+H_{\mu_E}(\cE)\le \sum_{n=1}^\infty \mu_E(E_n)n\log|\cQ|+H_{\mu_E}(\cE).
			\]
			Since both $H_{\mu_E}(\cE)$ and $\log|\cQ|\sum_{n=1}^\infty n\mu_E(E_n)$ are finite 
			we see that $H_{\mu_E}(\cQ_E^{(-\xi, 0]})<\infty$.
			By adding the set $P_0=X\setminus E$ as a cell to $\cQ_E^{(-\xi, 0]}$ we obtain a partition $\cQ_0^{(-\xi, 0]}$ of $X$ satisfying the assumptions of
			Proposition \ref{prop:main}. Applying that result we get
			\[
			\mu(E) \h_{\mu_E}(T_E,\cQ_E^{(-\xi, 0]})=\h_\mu(T,\cQ_0^{(-\xi, 0]}).
			\]
			 Recall that for $\mu$-almost every $x\in X$ there exists a strictly increasing sequence $(n_k)_{k\in\Z}$ of integers satisfying $n_{-1}< 0 \le n_0$ such that the $\cQ$-name $(x_i)_{i\in\Z}$ of $x$ can be divided into $x_{(n_{k-1},n_k]}$ for $k\in\mathbb{Z}$, where each block ends with a single starred symbol preceded by some number (possibly zero) of non-starred symbols. Therefore, given such a $\cQ$-name $x=(x_k)_{k\in\Z}$ and $j\in\mathbb{Z}$ we find  $k\in\mathbb{Z}$ such that $n_{k-1}<j\le n_k$ and we define
			\begin{equation}\label{eq:isom}
			y_j=\begin{cases}x_{(n_{k-1},n_k]},&\text{ if }j= n_k,\\
			0,&\text{ otherwise}.
			\end{cases}
			\end{equation}
			Since every block $x_{(n_{k-1},n_k]}$ corresponds to a cell of $\cQ_0^{(-\xi, 0]}$, the resulting sequence $(y_j)_{j\in\mathbb{Z}}$ is a valid $\cQ_0^{(-\xi, 0]}$-name. The transformation $(x_j)_{j\in\mathbb{Z}}\mapsto(y_j)_{j\in\mathbb{Z}}$ given by \eqref{eq:isom} is clearly an isomorphism\footnote{It is even a finitary code, that is, it has the following property: to determine a value of any entry in an output sequence, one should only examine finitely many coordinates of the source sequence, this finite number of coordinates depending upon the input sequence under consideration (see \cite{Serafin}).} of the processes $(T,\hat\cP_0^\cQ)$ and $(T,\hat\cQ)$, since given a $\cQ_0^{(-\xi, 0]}$-name, where nonzero blocks correspond to the blocks $x_{(n_{k-1},n_k]}$ in $\cQ_0^{(-\xi, 0]}$ we can easily reconstruct $\cQ$-name. Thus $\h_\mu(T,\cQ_0^{(-\xi, 0]})=\h_\mu(T,\cQ)$.
		\end{proof}

For the proof of Theorem \ref{thm:main-fbar} it will be convenient to replace $\fbar$ by a uniformly equivalent pseudometric $\hat f$.
For $u=(u_i)_{i\in\bbS},v=(v_i)_{i\in\bbS}\in\Lambda^{\bbS}$, where $\bbS=\mathbb{N}_0$ or $\bbS=\mathbb{Z}$ and strictly increasing sequences $I=(i(r))_{r\in\N}$,  $I'=(i'(r))_{r\in\N}$ in $\mathbb{N}_0$ we write $u|_{I}=w|_{I'}$ if
$u_{i(r)}=w_{i'(r)}$ for every $r\in\N$.  We define $\hat f(u,w)$ as
\begin{multline*}\hat f(u, w)=\inf\{\eps>0\,:u|_{I}=w|_{I'}\text{ for some strictly increasing sequences }\\I=(i(r))_{r\in\N},\,I'=(i'(r))_{r\in\N}\text{ in }\N_0\text{ with }\und d(I)\geq 1-\eps,\, \und d(I')\geq 1-\eps\},\end{multline*}
where $\underline d$ denotes the \emph{lower asymptotic density}, that is, for $A\subset\N_0$ we set
\[\und d(A)=\liminf\limits_{n\to\infty}\frac1n{|A\cap\{0,1,\ldots, n-1\}|}.\]


\begin{lemma}[\cite{ORW}]\label{fbarhat}
	The pseudometrics $\hat{f}$ and $\f$ are uniformly equivalent on $\Lambda^{\bbS}$.
\end{lemma}

	\begin{proof}[Proof of Theorem~\ref{thm:main-fbar}] By Lemma~\ref{fbarhat} it is enough to consider the  pseudometric $\hat f$, which is a pseudometric on $\Lambda^\bbS$ in both cases, $\bbS=\N_0$ and $\bbS=\Z$. Theorem~\ref{thm:main-fbar} follows immediately from the analogous statement for invertible processes by considering the natural extension, so from now on we assume that $\bbS=\Z$. Let $\Lambda=\{1,2,\dots,l\}$  ($\Lambda$ deliberately does not contain $0$).
		Let $x,x'\in \Lambda^{\Z}$ satisfy $\hat f(x,x')<\eps$. We also assume that $x,x'$ are frequency-typical for measures $\mu,\mu'$, respectively. We wish to prove that $|h_\mu(\sigma)-h_{\mu'}(\sigma)|\to 0$ as $\eps\to 0$.
		By assumption, there exist sets $A=\{a(1),a(2),\dots\}\subset\N_0$ and $A'=\{a'(1),a'(2),\dots\}\subset\N_0$ such that $x_{a(n)}=x'_{a'(n)}$ for each $n\in\mathbb{N}$, and both $\und d(A)$ and $\und d(A')$ are bounded below by $1-\eps$.
We define $\kappa=(\kappa_n)_{n=0}^\infty$ by $\kappa_n=x_{a(n+1)}=x'_{a'(n+1)}$ for $n\in\mathbb{N}_0$. Let $y\{0,1\}^\Z$ be the characteristic function of a set $A$. 		
			\begin{figure}[ht]
{\centering
				\includegraphics[scale=0.5]{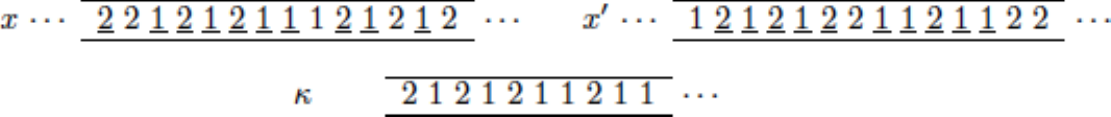}
				\caption{An example of sequences $x, x'\in\{1,2\}^{\Z}$ and their common subsequence $\kappa$. Underlined entries in $x$ and $x'$ mark the positions in $A$ and $A'$.} }
				\label{fig:2}
			\end{figure}
	
		
		

		

In the following, we will repeatedly choose $\sigma$-invariant measures quasi-generated by various finite valued sequences without specifying sequences along which these measures are generated. Each time we pass to a subsequence,  we  choose it to be a subsequence of the sequence along which the point used in the previous step quasi-generated a measure. For example, $x$ is frequency-typical, so it generates $\mu$ along the sequence consisting of all nonnegative integers, while $\xi'$ is quasi-generated along a subsequence of the sequence along which $\xi$ is generated. 
		
		
		Let $\xi$ be a $\sigma\times\sigma$-invariant measure on $\Lambda^\Z\times\{0,1\}^\Z$ quasi-generated by the pair $(x,y)$. Let $\nu$ on $\{0,1\}^\Z$ be the marginal distribution of $\xi$ on the second coordinate. Then $\xi$ is a joining of $\mu$ and $\nu$ (cf. \cite[p. 52]{Gray}) and $\nu$ is quasi-generated by $y$, that is, $\mu$ (respectively, $\nu$) is the marginal of $\xi$ with respect to the projection from $\Lambda^\Z\times\{0,1\}^\Z$ to the first (respectively, the second) coordinate. Note that the entropy rate of $\nu$ satisfies $h_{\nu}(\sigma)\le \HSh(\eps)+\HSh(1-\eps)$. 
		Thus,
		\[
		h_{\mu}(\sigma)\le h_\xi(\sigma\times\sigma)\le h_\mu(\sigma)+h_\nu(\sigma)\le h_\mu(\sigma)+\HSh(\eps)+\HSh(1-\eps).
		\]
		Therefore  $|h_{\mu}(\sigma)- h_\xi(\sigma\times\sigma)|\to 0$ as $\eps\to 0$.
			\begin{figure}[h]
{\centering
				\includegraphics[scale=0.5]{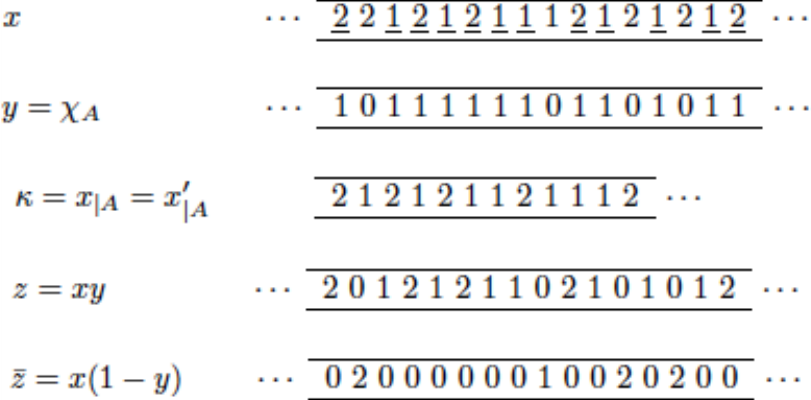}
				\caption{Sequences $y, \kappa, z, \bar z$ constructed as in the proof of Theorem \ref{thm:main-fbar} given $x, x'$ as above.}}
			\end{figure}

Let $z = xy\in(\Lambda\cup\{0\})^\Z$ be the pointwise product of $x$ and $y$ and let $\bar z = x(1-y)$. Note that $z$ (respectively, $\bar z$) coincides with $x$ along $A$ (respectively, along $A^c=\N_0\setminus A$). 
The finite sliding-block code of length $1$ given by 
$\Phi(a,b)=(ab,a(1-b))$ yields an isomorphism between the measure $\xi$ quasi-generated by $(x,y)$ and  a measure $\bar\xi$ 
quasi-generated by $(z,\bar z)$ along the same sequence as $\xi$. The marginal distributions of $\bar\xi$ are 
$\sigma$-invariant measures $\zeta$ and $\bar\zeta$ quasi-generated by $z$ and  $\bar z$, respectively. 
		Hence 
		\begin{equation}\label{eq:zeta}
		h_\zeta(\sigma)\le h_\xi(\sigma\times\sigma)=h_{\bar\xi}(\sigma\times\sigma)\le h_\zeta(\sigma)+h_{\bar\zeta}(\sigma).
		\end{equation}
		Note that $\bar\zeta$ is generated by $\bar z$ and the symbol $0$ appears in $\bar z$ with the lower asymptotic density bounded below by
$1-\eps$, so $\bar\zeta([0])\ge 1-\eps$. Similar reasoning shows that $\zeta([0])\le\eps$. It follows that $h_{\bar\zeta}(\sigma)\le \HSh(1-\eps)+\HSh(\eps)+\eps\log l$, where the right hand-side of the inequality is the entropy of the probability vector $(1-\eps,\eps/l,\eps/l,\dots,\eps/l)$, with $l=|\Lambda|$. This, together with \eqref{eq:zeta}, imply that
		$|h_\zeta(\sigma)- h_\xi(\sigma\times\sigma)|\to 0$ 
as $\eps\to0$.
		
It remains to estimate $h_\zeta(\sigma)$. 
Consider the measure preserving system $(X,\sX,\zeta,\sigma)$ where $X=(\Lambda\cup\{0\})^\Z$, $\sX$ is the $\zeta$-completion of the product $\sigma$-algbra, and $\sigma$ is the shift transformation.  Let $E=X\setminus[0]$ be the set of seqeunces $(x_n)_{n\in\Z}$ with $x_0\neq  0$. Note that $\zeta([0])\le\eps$ implies $\zeta(E)\ge 1-\eps$.
We consider the induced system $(E,\sX_E,\zeta_E,\sigma_E)$.
Let $\cQ=\{[a]:a\in\Lambda\cup\{0\}\}$ be the partition of $X$ into cylinder sets of length $1$. 	Then $\cQ_E=\{[a]:a\in\Lambda\}$, $\cQ=\cQ_E\cup\{[0]\}$, and  $\cQ\succcurlyeq\{E,X\setminus E\}$. 
We want to apply Proposition \ref{prop}, but $E$ need not to sweep out $X$ with respect to $\zeta$. 
In that case, there are $0<\alpha<1$ and a shift invariant measure $\hat\zeta$ such that $\zeta=(1-\alpha)\hat\zeta+\alpha\delta_{\mathbf{\bar{0}}}$, where $\delta_{\mathbf{\bar{0}}}$ is the Dirac measure concentrated on the fixed point $\mathbf{\bar{0}}=\ldots 00.00\ldots$ and  $\hat\zeta$ and $\delta_{\mathbf{\bar{0}}}$ are mutually singular.
It follows that $\hat\zeta(\{\mathbf{\bar{0}}\})=0$, hence $E$ sweeps out $X$ with respect to $\hat\zeta$.
Furthermore, $\alpha\le\zeta([0])\le\eps$ and
\[
		\hat\zeta(E)= \frac{\zeta(E)}{1-\alpha}\ge \frac{1-\eps}{1-\alpha}.
\]
		By affinity of the entropy rate \cite[Lemma 3.9]{Gray} we get $h_\zeta(\sigma)=(1-\alpha)h_{\hat\zeta}(\sigma)+\alpha h_{\delta_{\mathbf{\bar{0}}}}(\sigma)=(1-\alpha)h_{\hat\zeta}(\sigma)$.
		It follows that $|h_{\zeta}(\sigma)-h_{\hat\zeta}(\sigma)|\to 0$ as $\eps\to 0$.

Let $\cQ_E^{(-\xi,0]}$ be the partition of $E$ obtained in Proposition \ref{prop}. Since $X\setminus E\in \cQ$, we
easily see that the processes $(T_E,\cQ_E^{(-\xi,0]})$ and $(T_E,\cR)$ where $\cR=\cQ_E\vee\cE$ are isomorphic.
Applying Proposition \ref{prop} we obtain $H_{\hat\zeta_E}(\cR)<\infty$ and
		\[
		\hat\zeta(E) \h_{\hat\zeta_E}(\sigma_E, \cR)=\h_{\hat\zeta}(\sigma, \cR\cup\{[0]\})=\h_{\hat\zeta}(\sigma, \cQ)=h_{\hat\zeta}(\sigma).
		\]
It follows that $|\h_{\hat\zeta_E}(\sigma_E, \cR)-h_{\hat\zeta}(\sigma)
		|\to 0$ 
as $\eps\to 0$.

Consider the entropy rate of the process $(\sigma_E,\cE)$ 
generated 
on $(E,\sX_E,\hat\zeta_E,\sigma_E)$.  Setting
$p_n=\hat{\zeta}(E_n)=\hat{\zeta}\left(\{x\in E:r_E(\sigma_E^{-1}(x))=n\}\right)$ we get a probability vector $\mathbf{p}=(p_n)_{n\in\N}$ on $\N$ with the expected value
		\[
		\mathbb{E}(\mathbf{p})=\sum_{n=1}^\infty np_n =1/\hat\zeta(E)\le 1/\zeta(E).
		\]
Since it is well-known that among all the distributions on $\N$ with expected value $1/p$, the largest entropy is 
$(\HSh(p)+\HSh(1-p))/p$ we get $\h_{{\hat\zeta_E}}(\sigma_E, \cE)\le (\HSh(\eps)+\HSh(1-\eps))/(1-\eps)$. 
		Since $\cR=\cQ_E\vee\cE$ we have
$\h_{\hat\zeta_E}(\sigma_E, \cQ_E)\le \h_{\hat\zeta_E}(\sigma_E,\cR)$ and
		\[
 \h_{\hat\zeta_E}(\sigma_E, \cR)\le \h_{\hat\zeta_E}(\sigma_E, \cQ_E)+\h_{\hat\zeta_E}(\sigma_E, \cE)\le \h_{\hat\zeta_E}(\sigma_E, \cQ_E)+\frac{\HSh(\eps)+\HSh(1-\eps)}{(1-\eps)}.
		\]
		Thus $|\h_{\hat\zeta_E}(\sigma_E, \cQ_E)- \h_{\hat\zeta_E}(\sigma_E, \cR)|\to 0$ as $\eps\to0$.
Summing up, $|\h_{\hat\zeta_E}(\sigma_E, \cQ_E) - h_{\mu}(\sigma)|$ 
approaches $0$ as $\eps\to 0$.
		
		Repeating the same steps, but starting with $x'$ in place of $x$ and generating only along subsequences of the sequence along which $\zeta$ was generated, we produce ``primed versions'' of all objects defined so far. In particular, we have a measure $\hat\zeta'$
		such that taking the same set $E$ as above and the partition $\cQ_{E}$ of $E$ into cylinders of non-zero symbols
		$|\h_{\hat\zeta'_{E}}(\sigma_{E}, \cQ_{E}) - h_{\mu'}(\sigma)|\to0$ as $\eps\to 0$. It remains to compare $\h_{\hat\zeta_{E}}(\sigma_{E}, \cQ_{E})$ with its primed variant $\h_{\hat\zeta'_{E}}(\sigma_{E}, \cQ_{E})$.

We clearly have $\zeta_E=\hat\zeta_E$ and $\zeta'_E=\hat\zeta'_E$. Furthermore, 
$z$ restricted to nonzero entries with nonnegative indices, which coincides with $\kappa$  defined above, gives us a $\cQ_E$-name of a frequency-typical sequence for the induced process $(\sigma_E,\cQ_E)$.
		But $z'$ restricted to nonzero entries also yields $\kappa$ and is frequency-typical for the process $(\sigma_{E}, \cQ_{E})$, hence
		$\h_{\hat\zeta_{E}}(\sigma_{E}, \cQ_{E})=\h_{\hat\zeta'_{E}}(\sigma_{E}, \cQ_{E})$, and the proof is complete.
	\end{proof}

	\begin{rem}
		The above proof indicates also the modulus of the uniform continuity of the entropy function with respect to $\hat f$, namely if two frequency-typicasl sequences are $\eps$-close with respect to $\fbar$, then the entropy rates of the processes generated by these sequences differ by at most \[2\left(2\eta(\eps)+2\eta(1-\eps)+4\eps\log l+\frac{\eta(\eps)+\eta(1-\eps)}{1-\eps}\right).\]
		Note that this number depends on the cardinality $l$ of the alphabet $\Lambda$.
	\end{rem}
	\section{Appendix: The Abramov Formula}
	
	As a by-product of our considerations we present an elementary proof of a general version of the Abramov formula, which \cite{Krengel} attributes to Scheller.
Let $\Part(D)$ (respectively, $\Part_\omega(D)$) stand for the set of all finite measurable partitions of $D\in\sX$ (respectively, all countable  measurable partitions $\cP$ of $D$ with $H_\mu(\cP)<\infty$). Recall that the \emph{Kolmogorov-Sinai entropy} $h_\mu(T)$ of a measure preserving system $(X,\sX,\mu,T)$ is the supremum of entropy rates of all processes generated from the system by taking $\cP\in\Part(X)$, equivalently, by taking $\cP\in\Part_\omega(X)$, that is,
\begin{equation}\label{eq:entropy}
h_\mu(T)=\sup_{\cP\in\Part(X)}\h_\mu(T, \cP)=\sup_{\cP\in\Part_\omega(X)}\h_\mu(T, \cP).
\end{equation}

	\begin{theorem}[The Abramov Formula]\label{thm:Abramov}
		Let $(X,\sX,\mu,T)$ be a probability measure preserving system and let $E\in\sX$ sweep out $X$. 
		Then $h_\mu(T)=\mu(E) h_{\mu_E}(T_E)$.
	\end{theorem}
	\begin{proof}
Using the natural extension the proof reduces to the invertible case. For every $\cQ\in\Part(E)$, the partition $\cQ^{(E)}=(\cQ\vee \cE)\cup\{X\setminus E\}\in\Part_\omega(X)$ satisfies the assumptions of Proposition \ref{prop:main} and $\cQ^{(E)}_E\succcurlyeq \cQ$, so $\mu(E)\h_{\mu_E}(T_E,\cQ)\le h_{\mu}(T,\cQ^{(E)})$.
Using \eqref{eq:entropy} we get		\[
		\mu(E)\h_{\mu_E}(T_E)=\mu(E)\sup_{\cQ\in\text{Part}(E)} \h_{\mu_E}(T_E,\cQ) \le \sup_{\cQ\in\text{Part}(E)}\h_{\mu}(T,\cQ^{(E)})\le h_\mu(T).
		\]
For $\cQ\in\Part(X)$ we set $\hat\cQ=\cQ\vee\{E,X\setminus E\}$.  Let $\hat\cQ_E^{(-\xi,0]}$ be the partition of $E$ obtained in Proposition \ref{prop}.
%
We have
		\[
		h_\mu(T)=\sup_{\cQ\in\Part(X)} \h_\mu(T,\hat\cQ)= \mu(E)
\sup_{\cQ\in\Part(X)} \h_{\mu_E}(T_E,\hat\cQ^{(-\xi,0]})
		\le \mu(E) h_{\mu_E}(T_E),
\]
where the first equality uses that $\hat\cQ\succcurlyeq\cQ$	for every $\cQ\in\Part(X)$, the second equality follows from Proposition \ref{prop}, and the inequality follows from \eqref{eq:entropy}.
	\end{proof}
	\section{Acknowledgments}
The authors would like to thank Michal Kupsa for many discussions on the subject of the paper.
Martha \L{}\k{a}cka acknowledges the support of the National Science Centre (NCN), Poland, grant \emph{Preludium} no. 2015/19/N/ST1/00872 and the doctoral scholarship \emph{Etiuda} no. 2017/24/T/ST1/00372. Dominik Kwietniak acknowledges the support of the National Science Centre (NCN), Poland, grant \emph{Opus} no. 2018/29/B/ST1/01340. Tomasz Downarowicz acknowledges the support of the National Science Centre (NCN), Poland, grant \emph{Harmonia} no. 2018/30/M/ST1/00061. 	
The authors are greatly indebeted to the anonymous refrees for several comments which helped to improve the paper.

\Addresses


\bigskip\noindent

\begin{thebibliography}{xx}
\bibitem{CDS} J.-P. Conze, T. Downarowicz, J. Serafin, \emph{Correlation of sequences and of measures, generic points for joinings and ergodicity of certain cocycles}. Trans. Amer. Math. Soc. \textbf{369} (2017), no. 5, 3421--3441. 
\bibitem{DGS}	M. Denker, C. Grillenberger, K. Sigmund, \emph{Ergodic theory on compact spaces}.  Lecture Notes in Mathematics, Vol. 527. Springer-Verlag, Berlin-New York, 1976. 
\bibitem{Downarowicz} T.~Downarowicz,
\emph{Entropy in dynamical systems}, volume 18 of New Mathematical Monographs. Cambridge University Press, Cambridge, 2011.
\bibitem{Feldman} J. Feldman, \emph{New $K$-automorphisms and a problem of Kakutani}, Israel J. Math. \textbf{24} (1976), no. 1, pp. 16--38. 
\bibitem{Gray} Robert M. Gray, \emph{Entropy and information theory}. Second edition. Springer, New York, 2011. xxviii+409 pp. ISBN: 978-1-4419-7969-8; 978-1-4419-7970-4  
\bibitem{GNS} R. M. Gray, D. L. Neuhoff and P. C. Shields, \emph{A generalization of Ornstein's $\dbar$ distance with applications to Information Theory},
The Annals of Probability, 3, 2 (1975).
\bibitem{Kalikow} S. Kalikow, R. McCutcheon, \emph{An outline of ergodic theory.} Cambridge Studies in Advanced Mathematics, 122. Cambridge University Press, Cambridge, 2010. 
\bibitem{Katok} A. Katok, \emph{Monotone equivalence in ergodic theory} (Russian), Izv. Akad. Nauk SSSR Ser. Mat. \textbf{41} (1977), no. 1, 104--157, 231.
\bibitem{Krengel} U. Krengel, \emph{On certain analogous difficulties in the investigation of flows in a probability space and of transformations in an infinite measure space}. 1969 Functional Analysis (Proc. Sympos., Monterey, Calif., 1969) 75--91 Academic Press, New York. 
\bibitem{KL}  D. Kwietniak, M. \L{}\k{a}cka, {\emph{Feldman-Katok pseudometric and the GIKN construction of nonhyperbolic ergodic measures}}, arXiv:1702.01962, pp.~19.
\bibitem{KLO} D. Kwietniak, M.  \L{}\k{a}cka, P. Oprocha, {\emph{Generic Points for Dynamical Systems with Average Shadowing}}, Monatshefte f\"ur Mathematik 183 (4), 2017, pp.~625--648.
\bibitem{Ornstein}	D.S. Ornstein, 	\emph{An application of ergodic theory to probability theory}, Annals of
	Probability, 1, 1 (1973).
\bibitem{ORW} D. Ornstein, D. Rudolph, B. Weiss, \emph{Equivalence of measure preserving transformations}, Mem. Amer. Math. Soc. \textbf{37} (1982), no. 262. 
\bibitem{Petersen} Karl Petersen, \emph{Ergodic theory}. Cambridge Studies in Advanced Mathematics, 2. Cambridge University Press, Cambridge, 1983. 
\bibitem{Rudolph} Daniel J. Rudolph, \emph{Fundamentals of measurable dynamics. Ergodic theory on Lebesgue spaces}. 
Oxford University Press, New York, 1990. 
\bibitem{Saussol}  Benoit Saussol, \emph{An introduction to quantitative Poincar\'{e} recurrence in dynamical systems}, Rev. Math. Phys. 21 (2009), no. 8, 949--979.
\bibitem{Serafin} J. Serafin, \emph{Finitary codes, a short survey.} Dynamics \& stochastics, 262--273, IMS Lecture Notes Monogr. Ser., 48, Inst. Math. Statist., Beachwood, OH, 2006.
\bibitem{Shields} Paul C. Shields. The ergodic theory of discrete sample paths, 
American Mathematical Society, Providence, RI, 1996.
\bibitem{Shannon1} C. E. Shannon, \emph{A Mathematical Theory of Communication}. Bell System Technical Journal, Volume 27, Issue 3, July 1948.
\bibitem{Shannon2} C. E. Shannon, \emph{Coding theorems for a discrete source with a fidelity criterion}. Institute of Radio Engineers, International Convention Record, vol. 7, 1959.
\bibitem{Vershik} A. M.~Vershik, The Kantorovich metric: the initial history and little-known applications.
Zap. Nauchn. Sem. S.-Peterburg. Otdel. Mat. Inst. Steklov. (POMI) 312 (2004), Teor. Predst. Din. Sist. Komb. i Algoritm. Metody. 11, 69--85, 311; translation in J. Math. Sci. (N.Y.) 133 (2006), no. 4, 1410--1417.
\bibitem{Villani} C. Villani, Optimal transport. Old and new. Grundlehren der Mathematischen Wissenschaften 
338. Springer-Verlag, Berlin, 2009. 
\bibitem{Weiss} B. Weiss, \emph{Single orbit dynamics}, CBMS Regional Conference Series
in Mathematics, vol.~95, American Mathematical Society, Providence, RI, 2000. 
\end{thebibliography}
\end{document}